\documentclass[12pt]{amsart}

\usepackage{graphicx}
\usepackage{times}

\newtheorem{thm}{Theorem}

\newtheorem{lem}[thm]{Lemma}
\newtheorem{prop}[thm]{Proposition}
\newtheorem{rem}[thm]{Remark}

\begin{document}

\title[Semi-linear Torsional Rigidity on Complete Riemannian 2-Manifold]
{The Semi-linear Torsional Rigidity on a Complete Riemannian
Two-Manifold}

\author{Jie Xiao}
\address{Department of Mathematics and Statistics, Memorial University of Newfoundland,
St. John's, NL A1C 5S7, Canada}
\email{jxiao@mun.ca}

\thanks{The project was supported in part by NSERC of Canada}

\begin{abstract}
This note is concerned with some essential properties (optimal
isoperimetry, first variation, and monotonicity formula) of the
so-called $[0,1)\ni\gamma$-torsional rigidity
$\mathcal{T}_{\gamma,\mathsf{g}}$ on a complete Riemannian
two-manifold $(\mathbb M^2,\mathsf{g})$. Even in the special case of
$\mathbb R^2$, major results are new.
\end{abstract}
\maketitle

\section{Introduction}

Throughout this note, on $(\mathbb M^2,\mathsf{g})$ -- a
two-dimensional manifold $\mathbb M^2$ with a complete Riemannian
metric $\mathsf{g}$, we denote by
$$
d_\mathsf{g}(\cdot,\cdot);\ \ \langle\cdot,\cdot\rangle_\mathsf{g};\
\ |\cdot|_{\mathsf{g}};\ \ K_{\mathsf g}(\cdot,\cdot);\ \
dA_{\mathsf g}(\cdot);\ \ dL_{\mathsf g}(\cdot);\ \ \Delta_{\mathsf
g}(\cdot);\ \ \nabla_{\mathsf g}(\cdot),
$$
the distance function; the inner product between two vectors in the
tangent bundle; the norm of a vector; the Gauss curvature; the area
element; the length element; the Laplace-Beltrami operator; the
gradient, respectively. Moreover, $B_\mathsf{g}(o,r)=\{z\in\mathbb
M^2: d_\mathsf{g}(z,o)<r\}$ denotes the geodesic disk centered at
$o$ with radius $r$, and the isoperimetric constant of $(\mathbb
M^2,\mathsf{g})$ is determined by
$$
\tau_{\mathsf g}=\inf_{O\in \mathcal{F}(\mathbb
M^2)}\frac{\big(L_{\mathsf{g}}(\partial
O)\big)^2}{A_{\mathsf{g}}(O)}.
$$
When $\mathbb M^2$ is the flat Euclidean plane $\mathbb R^2$, we
naturally equip it with the standard Euclidean metric $\mathsf{e}$
and therefore the previous notations will be changed
correspondingly, i.e., $\mathsf{g}$ is replaced by $\mathsf{e}$. In
particular, $\tau_\mathsf{e}=4\pi$.

For a parameter $\gamma \in [0,1)$ and a relatively compact domain
$O\subseteq\mathbb M^2$ with $C^\infty$ smooth boundary $\partial
O$, denoted by $O\in \mathcal{F}(\mathbb M^2)$, let $u$ be the
solution of the following semi-linear boundary value problem (see
\cite{Spe}, \cite{Co}, \cite{DrKe}, \cite{DaHeHu}, \cite{CarRa}, and
their related references for the Euclidean case $\mathbb R^2$):
\begin{equation}\label{eq1}
\left\{\begin{array} {r@{\quad\quad}l}
\Delta_{\mathsf{g}} u=-u^\gamma\ \ \&\ \ u>0 & \hbox{in}\quad O;\\
u=0 & \hbox{on}\quad\partial O,
\end{array}
\right.
\end{equation}
where the second identity follows from Green's theorem. Then the
semi-linear (or $\gamma$-) torsional rigidity of $O$ as the cross
section of the cylindrical beam $O\times\mathbb R$ is defined as
\begin{equation*}\label{}
\mathcal{T}_{\gamma,\mathsf{g}}(O)=\int_O\big|\nabla_{\mathsf
g}u\big|_\mathsf{g}^2\,dA_{\mathsf g}=\int_{O}
u^{1+\gamma}\,dA_{\mathsf g}.
\end{equation*}
Note that if $\gamma=0$ then (\ref{eq1}) is just the classical
torsion problem and the resulting $0$-torsional rigidity is
standard. As well-known, under $\gamma=1$ the problem (\ref{eq1})
has more than one non-trivial solutions, and thus the following
eigenvalue problem is instead considered:
\begin{equation}\label{eq1a}
\left\{\begin{array} {r@{\quad\quad}l}
\Delta_{\mathsf{g}} u=-\lambda u\ \ \&\ \ u>0 & \hbox{in}\quad O;\\
u=0 & \hbox{on}\quad\partial O,
\end{array}
\right.
\end{equation}
whose principal (or first) eigenvalue is determined through
$$
\Lambda_{\mathsf{g}}(O):=\inf_{v\in
W^{1,2}_0(O)}\left\{\int_{O}\big|\nabla_\mathsf{g}
v\big|_\mathsf{g}^2\,dA_\mathsf{g}:\quad
\int_{O}v^{2}\,dA_\mathsf{g}=1\right\},
$$
where $W^{1,2}_0(O)$ stands for the Sobolev space of all
compactly-supported $C^\infty$ functions $v$ on $O$ with $v^2$ and
$|\nabla_\mathsf{g} v|^2_\mathsf{g}$ being
$dA_\mathsf{g}$-integrable on $O$.

On the basis of Section \ref{s5} -- a $\gamma$-torsional rigidity
Schwarz's lemma for the conformal mappings on $\mathbb R^2$, we
shall present some fundamental properties of
$\mathcal{T}_{\gamma,\mathsf{g}}$ in: Section \ref{s2} -- the
optimal isoperimetric inequality in terms of $\tau_{\mathsf{g}}$;
Section \ref{s3} -- the first variational formula arising from a
domain deformation; Section \ref{s4} -- the monotonicity for the
$\gamma$-torsional rigidity of a geodesic disk.

\section{Isoperimetry}\label{s2}

Whenever $\mathbb M^2=\mathbb R^2$, a famous problem posed by St.
Venant in 1956 and settled by G. P\'olya in 1948 (cf. \cite[p.
121]{PoSz}) was to prove that among all simply connected domains of
given area, a disk of the area has the largest $0$-torsional
rigidity. Such an isoperimetric result can be naturally extended to
the $\gamma$-torsional rigidity.

\begin{prop}\label{pr2} Given $\gamma\in [0,1)$. Let $(\mathbb M^2,\mathsf{g})$ be a complete Riemannian two-manifold
with $\tau_\mathsf{g}>0$. If $u$ is the solution of (\ref{eq1}) with
$O\in\mathcal{F}(\mathbb M^2)$ being simply-connected, then
\begin{equation}\label{eq2}
\int_O u^{1+\gamma}\,dA_{\mathsf
g}\le\Big(\frac{1+\gamma}{2\tau_\mathsf{g}}\Big)\left(\int_O
u^\gamma\,dA_{\mathsf g}\right)^2,
\end{equation}
equivalently,
\begin{equation}\label{eq22}
\int_O \big|\nabla_\mathsf{g}
u\big|_\mathsf{g}^2\,dA_\mathsf{g}\le\Big(\frac{1+\gamma}{2\tau_\mathsf{g}}\Big)\left(\int_{\partial
O}\big|\nabla_\mathsf{g} u\big|_\mathsf{g}\,dL_{\mathsf g}\right)^2.
\end{equation}
Moreover, if $\mathbb M^2=\mathbb R^2$ and $O=B_\mathsf{g}(o,r)$,
then equality of (\ref{eq2}) or (\ref{eq22}) is valid.
\end{prop}
\begin{proof} Partially inspired by R. Sperb's exposition in \cite[pp.
190-196]{Spe}, we make the following argument.

Given a simply-connected domain $O\in\mathcal{F}(\mathbb M^2)$. For
$0\le t\le S:=\sup_{z\in O}u(z)$ let
$$
O_t=\{z\in O: u(z)>t\};\ \ \partial O_t=\{z\in O: u(z)=t\};\ \
a(t)=A_\mathsf{g}(O_t).
$$
Without loss of generality, we may assume that the set of the
critical points of $u$ is finite. An application of the well-known
co-area formula gives
\begin{equation}\label{1e}
\frac{da(t)}{dt}=-\int_{\partial O_t}\big|\nabla_\mathsf{g}
u\big|^{-1}_\mathsf{g}\,dL_\mathsf{g}.
\end{equation}
Using (\ref{1e}), Cauchy-Schwarz's inequality and
$\tau_\mathsf{g}>0$, we find
\begin{equation}\label{2e}
\tau_\mathsf{g}a(t)\le\big(L_\mathsf{g}(\partial
O_t)\big)^2\le\Big(-\frac{da(t)}{dt}\Big)\int_{\partial
O_t}\big|\nabla_\mathsf{g}u\big|_\mathsf{g}\,dL_\mathsf{g}.
\end{equation}

For convenience, set
$$
I_\gamma(t)=\int_{O_t}u^\gamma\,dA_\mathsf{g}\quad\&\quad
I_{1+\gamma}(t)=\int_{O_t}u^{1+\gamma}\,dA_\mathsf{g}.
$$
Then, using the layer-cake formula, the integration-by-part and
(\ref{1e}), we get
$$
I_\gamma(t)=\int_t^{S}\Big(\int_{\partial O_s}\big|\nabla_\mathsf{g}
u\big|^{-1}_\mathsf{g}\,dL_\mathsf{g}\Big)s^\gamma\,ds,
$$
whence finding
$$
\frac{dI_\gamma(t)}{dt}=-t^\gamma \int_{\partial
O_t}\big|\nabla_\mathsf{g}
u\big|^{-1}_\mathsf{g}\,dL_\mathsf{g}=t^\gamma\Big(\frac{da(t)}{dt}\Big)
$$
and so
\begin{equation}\label{e2e}
\frac{dI_\gamma(t)}{da(t)}=t^\gamma.
\end{equation}

On the other hand, an application of (\ref{2e}), Green's formula,
(\ref{eq1}), and $\tau_\mathsf{g}>0$ implies
\begin{equation}\label{e3e}
I_\gamma(t)=-\int_{O_t}\Delta_\mathsf{g}u\,dA_\mathsf{g}=\int_{\partial
O_t}\big|\nabla_\mathsf{g} u\big|_\mathsf{g}\,dL_\mathsf{g}\ge
\tau_\mathsf{g}a(t)\Big(-\frac{dt}{da(t)}\Big).
\end{equation}
By (\ref{e2e})-(\ref{e3e}) we obtain
\begin{equation}\label{yy}
I_\gamma(t)\Big(\frac{dI_\gamma(t)}{da(t)}\Big)+\tau_\mathsf{g}t^\gamma
a(t)\Big(\frac{dt}{da(t)}\Big)\ge 0.
\end{equation}
Now, choosing $a=a(t)$ as an independent variable, we get $A=a(0)$
and $0=a(S)$. Then, integrating (\ref{yy}) over the interval
$(0,A)$, taking an integration-by-part, and using (\ref{1e}) once
again, as well as the layer-cake formula, we achieve
\begin{eqnarray*}
0&\le& \int_0^A
\Big(\frac{dI_\gamma}{da}\Big)I_\gamma\,da+\tau_\mathsf{g}\int_0^A a
t^\gamma\Big(\frac{dt}{da}\Big)\,da\\
&=&2^{-1}\int_0^A
dI_\gamma^2-\Big(\frac{\tau_\mathsf{g}}{1+\gamma}\Big)\int_0^A
t^{1+\gamma}\,da\\
&=&
2^{-1}\big(I_\gamma(0)\big)^2-\Big(\frac{\tau_\mathsf{g}}{1+\gamma}\Big)\int_0^{S}
t^{1+\gamma}\Big(\int_{\partial O_t}\big|\nabla_\mathsf{g} u\big|_\mathsf{g}^{-1}\,dL_\mathsf{g}\Big)\,dt\\
&=&2^{-1}\big(I_\gamma(0)\big)^2-\Big(\frac{\tau_\mathsf{g}}{1+\gamma}\Big)I_{1+\gamma}(0),
\end{eqnarray*}
thereby finding (\ref{eq2}) right away.

Clearly, (\ref{eq22}) follows from (\ref{eq2}) and
$$
\int_O
u^\gamma\,dA_\mathsf{g}=-\int_O\Delta_\mathsf{g}u\,dA_\mathsf{g}=-\int_{\partial
O}\frac{\partial u}{\partial\nu}\,dL_\mathsf{g}=\int_{\partial
O}|\nabla_\mathsf{g} u|\,dL_\mathsf{g}
$$
in which the Green formula has been used and $\partial/\partial\nu$
represents the partial derivative along the unit outward normal to
the boundary $\partial O$.

The equality case of (\ref{eq2}) or (\ref{eq22}) under $\mathbb
M^2=\mathbb R^2$ and $O=B_\mathsf{e}(o,r)$ (the origin-centered disk
of radius $r$) can be verified via a direct calculation with the
radial solution $u$ (cf. \cite{DaHeHu}) to
\begin{equation*}\label{}
\left\{\begin{array} {r@{\quad\quad}l}
\Delta_{\mathsf{e}} u=-\kappa_\gamma u^\gamma\ \ \&\ \ u>0\ \ \hbox{in}\ \ B_\mathsf{e}(o,r);\\
u|_{\partial B_\mathsf{e}(o,r)}=0\ \ \hbox{and}\ \
\int_{B_\mathsf{e}(o,r)}u^{1+\gamma}\,dA_\mathsf{e}=1,
\end{array}
\right.
\end{equation*}
where
$$
\kappa_\gamma:=\inf_{v\in
W^{1,2}_0(B_\mathsf{e}(o,r))}\left\{\int_{B_\mathsf{e}(o,r)}\big|\nabla_\mathsf{e}
v\big|_\mathsf{e}^2\,dA_\mathsf{e}:\quad
\int_{B_\mathsf{e}(o,r)}|v|_\mathsf{e}^{1+\gamma}\,dA_\mathsf{e}=1\right\}.
$$
\end{proof}

\begin{rem}\label{re3} Under the same hypothesis on $(\mathbb M^2,\mathsf{g})$ as Proposition \ref{pr2}, we can
discover two interesting facts:

\item{\rm (i)} If $\gamma=0$, $K_\mathsf{g}\ge 0$, and
$$
\inf_{(o,r)\in\mathbb M^2\times
(0,\infty)}\frac{2\tau_\mathsf{g}\mathcal{T}_{0,\mathsf{g}}\big(B_\mathsf{g}(o,r)\big)}
{(\pi r^2)^2}\ge 1
$$
which, plus the special case $\gamma=0$ of (\ref{eq2}), implies
$$
\inf_{(o,r)\in\mathbb M^2\times (0,\infty)}\frac{A_{\mathsf
g}\big(B_\mathsf{g}(o,r)\big)}{\pi r^2}\ge 1,
$$
then $\mathbb M^2$ is isometric to $\mathbb R^2$ due to E. Hebey's
explanation on \cite[p. 244]{He}.

\item{\rm (ii)} When $\gamma=1$, the corresponding formulation of (\ref{eq2}) (cf. \cite[p. 195,
(11.24)]{Spe} for $\mathbb M^2=\mathbb R^2$) is: if $u$ denotes the
Laplace-Beltrami eigenfunction associated to
$\Lambda_\mathsf{g}(O)$, then

\begin{equation}\label{eq2a}
\int_O
u^2\,dA_\mathsf{g}\le\frac{\Lambda_\mathsf{g}(O)}{\tau_\mathsf{g}}\left(\int_O
u\,dA_{\mathsf g}\right)^2,
\end{equation}
amounting to,
\begin{equation}\label{eq2b}
\int_O \big|\nabla
u\big|_\mathsf{g}^2\,dA_\mathsf{g}\le\frac{1}{\tau_\mathsf{g}}\left(\int_{\partial
O}\big|\nabla_\mathsf{g} u\big|_\mathsf{g}\,dL_{\mathsf g}\right)^2.
\end{equation}
Moreover, equality in (\ref{eq2a}) or (\ref{eq2b}) holds for
$\mathbb M^2=\mathbb R^2$ and $O=B_\mathsf{e}(o,r)$.
\end{rem}

\section{Variation}\label{s3}

Following the first variation formula of the principal eigenvalue
(i.e., $\gamma=1$) discovered in P. R. Garabedian and M. Schiffer
\cite{GS} when $\mathbb M^2=\mathbb R^2$ and in A. El Soufi and S.
Ilias \cite{EI} for the general setting which was reformulated by F.
Pacard and P. Sicbaldi in \cite[Proposition 2.1]{PaSi}, we can
establish an extension from $\Lambda_\mathsf{g}$ to
$\mathcal{T}_{\gamma,\mathsf{g}}$ with $\gamma\in [0,1)$.

\begin{prop}\label{pr3} Let $\gamma\in [0,1)$ and $(\mathbb M^2,\mathsf{g})$ be a complete Riemannian two-manifold. For a given time interval $|t|<t_0$
suppose that $O_t=\xi(t,O_0)$ is the flow on a domain
$O_0\in\mathcal{F}(\mathbb M^2)$ associated to the vector field
$\Xi(t,z)$, i.e,
\begin{equation}\label{e3}
\left\{\begin{array} {r@{\quad\quad}l}
\partial_t(t,z)=\Xi\big(\xi(t,z)\big);\\
\xi(0,z)=z\in O_0.
\end{array}
\right.
\end{equation}
If $u_t$ is the solution of (\ref{eq1}) with $O$ replaced by $O_t$
and $\nu_t$ is the unit outward normal vector field to $\partial
O_t$, then
\begin{equation}\label{e4}
\frac{d}{dt}\mathcal{T}_{\gamma,\mathsf{g}}(O_t)\Big|_{t=0}=\Big(\frac{1+\gamma}{1-\gamma}\Big)\int_{\partial
O_0}\langle\nabla_\mathsf{g}
u_0,\nu_0\rangle^2_\mathsf{g}\langle\nabla_\mathsf{g}
\Xi,\nu_0\rangle_\mathsf{g}\,dL_\mathsf{g}.
\end{equation}
\end{prop}
\begin{proof} Note that $u_t\big(\xi(t,z)\big)=0$ holds for any
$z\in \partial O_0$. So, a differentiation with respect to $t=0$
gives $\partial_t
u_0\big|_{t=0}=-\langle\nabla_\mathsf{g}u_0,\Xi\rangle_\mathsf{g}$
on $\partial O_0$. Because $u_0$ vanishes on $\partial O_0$, only
the normal component of $\Xi$ plays a role in the last formula. As a
result, one gets
\begin{equation}\label{e5}
\partial_t
u_0\big|_{t=0}=-\langle\nabla_\mathsf{g}u_0,\nu_0\rangle_\mathsf{g}=\langle\Xi,\nu_0\rangle_\mathsf{g}\quad\hbox{on}\quad\partial
O_0.
\end{equation}

Next, since $-\Delta_\mathsf{g}u_t=u_t^\gamma$ holds in $O_t$,
taking the partial derivative of this last equation at $t=0$ yields
\begin{equation}\label{e6}
0=\Delta_\mathsf{g}\partial_{t} u_0\big|_{t=0}+\gamma
u_0^{\gamma-1}\partial_t u_0\big|_{t=0}\quad\hbox{in}\quad O_0.
\end{equation}

Now, an application of the definition of
$\mathcal{T}_{\gamma,\mathsf{g}}(O_t)$, (\ref{e5}), (\ref{e6}),
(\ref{eq1}) with $O_0$, and Green's formula derives
\begin{eqnarray*}
\frac{d}{dt}\mathcal{T}_{\gamma,\mathsf{g}}(O_t)\Big|_{t=0}&=&({\gamma+1})\int_{O_0}u^\gamma\partial_t
u_0\big|_{t=0}\,dA_\mathsf{g}\\
&=&\Big(\frac{\gamma+1}{\gamma-1}\Big)\int_{O_0}\Big(\partial_t
u_0\big|_{t=0}\Delta_\mathsf{g}u_0-u_0\Delta_\mathsf{g}\partial_t
u_0\big|_{t=0}\Big)\,dA_\mathsf{g}\\
&=&\Big(\frac{1+\gamma}{1-\gamma}\Big)\int_{\partial
O_0}\langle\nabla_\mathsf{g}
u_0,\nu_0\rangle^2_\mathsf{g}\langle\nabla_\mathsf{g}
\Xi,\nu_0\rangle_\mathsf{g}\,dL_\mathsf{g}.
\end{eqnarray*}
Finally, (\ref{e4}) follows.
\end{proof}

\begin{rem} Two comments are in order:

\item{\rm (i)} Under $\mathbb M^2=\mathbb R^2$ and $\gamma=0$, an early form of (\ref{e4}) was established by J. Hadamard
\cite{Had} (cf. \cite{Her}), but also a convex-body-based variant of
(\ref{e4}) was stated in A. Colesanti \cite[Proposition 18]{Co}.

\item{\rm (ii)} Clearly, (\ref{e4}) does not allow $\gamma=1$ whose
corresponding formula for the principal eigenvalue is the following:
(cf. \cite[Proposition 2.1]{PaSi}):
\begin{equation}\label{eLast}
\frac{d}{dt}\Lambda_{\mathsf{g}}(O_t)\Big|_{t=0}=-\int_{\partial
O_0}\langle\nabla_\mathsf{g}
u_0,\nu_0\rangle_\mathsf{g}^2\langle\nabla_\mathsf{g}
\Xi,\nu_0\rangle_\mathsf{g}\,dL_\mathsf{g}.
\end{equation}
Of course, $O_t$ in (\ref{eLast}) is generated by the solution $u_t$
of (\ref{eq1a}) with $\lambda$ replaced by
$\Lambda_{\mathsf{g}}(O_t)$.
\end{rem}

\section{Monotonicity}\label{s4}

According to \cite[p. 132]{Co}, we have that if $\mathbb M^2=\mathbb
R^2$, $\mathsf{g}=\mathsf{e}$, and $O$ is a convex domain containing
the origin in its interior, then
$v_r(z)=r^\frac{2}{1-\gamma}u(r^{-1}z)$ solves (\ref{eq1}) with $O$
replaced by its $r$-dilation $rO$ and hence
\begin{equation}\label{an}
\mathcal{T}_{\gamma,\mathsf{e}}(rO)=\int_{rO}|\nabla_\mathsf{e}v|_\mathsf{e}^2\,dA_\mathsf{e}=r^\frac{4}{1-\gamma}\int_O|\nabla_\mathsf{e}u|_\mathsf{e}^2\,dA_\mathsf{e}=r^\frac{4}{1-\gamma}
\mathcal{T}_{\gamma,\mathsf{e}}(O).
\end{equation}
This observation leads to the following monotonicity formula for the
$\gamma$-torsional rigidity of a geodesic disk.

\begin{prop}\label{pr3a} Given $\gamma\in [0,1)$. Let $(\mathbb M^2,\mathsf{g})$ be a complete Riemannian two-manifold with $K_\mathsf{g}\ge 0$ and
$\tau_\mathsf{g}>0$. If $o\in\mathbb M^2$ is fixed, then
$$
r\mapsto\mathcal{Q}_{\gamma,\mathsf{g}}(o,r)
:=\frac{\mathcal{T}_{\gamma,\mathsf{g}}(B_\mathsf{g}(o,r))}{r^\frac{\tau_\mathsf{g}}{\pi(1-\gamma)}}
$$
is monotone increasing in $(0,\infty)$. Consequently,
$$
\lim_{r\downarrow
0}\mathcal{Q}_{\gamma,\mathsf{g}}(o,r)\le\mathcal{Q}_{\gamma,\mathsf{g}}(o,r)\le\lim_{r\uparrow\infty}\mathcal{Q}_{\gamma,\mathsf{g}}(o,r)\quad\forall\quad
r\in (0,\infty)
$$
holds with equalities for $\mathbb M^2=\mathbb R^2$.
\end{prop}
\begin{proof} Suppose that $u$ is the solution of (\ref{eq1}) with
$O=B_\mathsf{g}(o,r)$. Since $K_\mathsf{g}\ge 0$, a generalized
version of the well-known Bishop-Gromov comparison theorem (cf.
\cite[p. 41, Theorem 2.14]{Pi}) yields
\begin{equation}\label{e33a}
\frac{d}{dr}\Big(r^{-1}{L_\mathsf{g}\big(\partial
B_\mathsf{g}(o,r)\big)}\Big)\le 0\quad\&\quad
L_\mathsf{g}\big(\partial B_\mathsf{g}(o,r)\big)\le 2\pi r.
\end{equation}
Applying $\tau_\mathsf{g}>0$, (\ref{eq22}), Green's formula,
Cauchy-Schwarz's inequality, and (\ref{e33a}), we get
\begin{eqnarray}\label{e33b}
\mathcal{T}_{\gamma,\mathsf{g}}\big(B_\mathsf{g}(o,r)\big)&\le&\Big(\frac{1+\gamma}{2\tau_\mathsf{g}}\Big)\Big(\int_{\partial
B_\mathsf{g}(o,r)}|\nabla_\mathsf{g}
u|_\mathsf{g}\,dL_\mathsf{g}\Big)^2\nonumber\\
&\le&\Big(\frac{1+\gamma}{2\tau_\mathsf{g}}\Big)L_\mathsf{g}\big(\partial
B_\mathsf{g}(o,r)\big)\int_{\partial
B_\mathsf{g}(o,r)}|\nabla_\mathsf{g}
u|^2_\mathsf{g}\,dL_\mathsf{g}\\
&\le&\Big(\frac{1+\gamma}{(\pi
r)^{-1}\tau_\mathsf{g}}\Big)\int_{\partial
B_\mathsf{g}(o,r)}|\nabla_\mathsf{g}
u|^2_\mathsf{g}\,dL_\mathsf{g}\nonumber
\end{eqnarray}

On the other hand, consider a vector field induced by a normal shift
$\delta\nu$, counted positively in the direction of the outward
normal to $\partial B_\mathsf{g}(o,r)$. More precisely, for $t>-r$
and $z\in\partial B_\mathsf{g}(o,r)$ let $\xi=\xi(t,z)$ be the point
on the geodesic radius starting at $o$ of $B_\mathsf{g}(o,r)$ with
$d_\mathsf{g}(o,\xi)=(1+tr^{-1})d_\mathsf{g}(o,z)$. Consequently, if
$B_\mathsf{g}(o,r)$ is chosen as the initial domain $O_0$ in
Proposition \ref{pr3}, then
$$
\xi(0, B_\mathsf{g}(o,r))=O_0\quad\&\quad \xi(t,
B_\mathsf{g}(o,r))=O_t=B_\mathsf{g}(o,r+t).
$$
Once setting $\Xi(\xi(t,z))$ be the point on the geodesic (radial)
direction from $o$ to $\xi(t,z)$ with
$(r+t)^{-1}d_\mathsf{g}(o,\xi)$ as its distance from $o$, we see
that (\ref{e3}) holds. Obviously, the unit outward normal vector to
the boundary $\partial O_0$ at $\xi\in\partial O_0$ is the unit
vector formed by $\xi$ and so equal to $\Xi(\xi)$. Suppose now that
$u$ is the solution of (\ref{eq1}) with $O=B_\mathsf{g}(o,r)$. Then,
an application of (\ref{e4}) gives
\begin{equation}\label{e33c}
\frac{d}{dr}\mathcal{T}_{\gamma,\mathsf{g}}\big(B_\mathsf{g}(o,r)\big)=\Big(\frac{1+\gamma}{1-\gamma}\Big)\int_{\partial
B_\mathsf{g}(o,r)}\big|\nabla_\mathsf{g}u\big|_\mathsf{g}\,dL_\mathsf{g}.
\end{equation}

Next, we employ (\ref{e33b}) and (\ref{e33c}) to achieve
$$
\frac{d}{dr}\mathcal{Q}_{\gamma,\mathsf{g}}(r)=\frac{r\frac{d}{dr}\mathcal{T}_{\gamma,\mathsf{g}}\big(B_\mathsf{g}(o,r)\big)
-\big(\frac{\tau_\mathsf{g}}{\pi(1-\gamma)}\big)\mathcal{T}_{\gamma,\mathsf{g}}\big(B_\mathsf{g}(o,r)\big)}{r^{1-\frac{\tau_\mathsf{g}}{\pi(1-\gamma)}}}\ge
0,
$$
thereby reaching the desired monotonicity. Of course, the
consequence part is immediate.
\end{proof}

\begin{rem} When $\gamma=1$, by (\ref{eLast}) and the foregoing proof we can
establish that under the same hypothesis on $(\mathbb
M^2,\mathsf{g})$ as in Proposition \ref{pr3a},
$$
r\mapsto
\mathcal{Q}_{\mathsf{g}}(o,r):=\frac{\Lambda_{\mathsf{g}}\big(B_\mathsf{g}(o,r)\big)}{r^{-\frac{\tau_\mathsf{g}}{2\pi}}}
$$
is monotone decreasing in $(0,\infty)$. Consequently,
$$
\lim_{r\uparrow
\infty}\mathcal{Q}_{\mathsf{g}}(o,r)\le\mathcal{Q}_{\mathsf{g}}(o,r)\le\lim_{r\downarrow
0}\mathcal{Q}_{\mathsf{g}}(o,r)\quad\forall\quad r\in (0,\infty)
$$
holds with equalities for $\mathbb M^2=\mathbb R^2$ -- this follows
immediately from the well-known fact (see e.g. \cite[p. 110]{Co})
that $\Lambda_{\mathsf{e}}$ is homogeneous of order $-2$.
\end{rem}

\section{Appendix}\label{s5}

In their 2008 paper \cite{BMMPR}, R. Burckel, D. Marshall, D. Minda,
P. Poggi-Corradini and T. Ransford discovered the
area-capacity-diameter versions of the following Schwarz's lemma
variant: For a holomorphic map $f$ from the origin-centered unit
disk $B_\mathsf{e}(o,1)$ into $\mathbb R^2$,
$$
r\mapsto \frac{\sup_{z\in
B_\mathsf{e}(o,r)}|f(z)-f(o)|_\mathsf{e}}{r}
$$
is strictly increasing in $(0,1)$ unless $f$ is linear. Soon after,
their results were extended differently by A. Y. Solynin \cite{Sol},
D. Betsakos \cite{Bet1}-\cite{Bet2}, and J. Xiao and K. Zhu
\cite{XZ}. While, as a new complement to \cite{BMMPR}, T. Carroll
and J. Ratzkin's 2010 article \cite{CarRa} on the Schwarz type lemma
for $\Lambda_\mathsf{e}$ has partially stimulated us to carry out
our current project. In contrast to the
monotone-decreasing-principle (i.e., the backward Schwarz type
lemma) in \cite{CarRa} saying that
$$
r\mapsto\frac{\Lambda_\mathsf{e}\big(f(B_\mathsf{e}(o,r))\big)}{\Lambda_\mathsf{e}(B_\mathsf{e}(o,r))}
$$
is strictly decreasing in $(0,1)$ unless $f$ is a linear map, we
have the forward Schwarz type lemma for the $\gamma$-torsional
rigidity:

\begin{lem}\label{pr4} Given $\gamma\in [0,1)$. If $f$ is a conformal mapping from $B_\mathsf{e}(o,1)$ into $\mathbb R^2$, then
$$
r\mapsto
\Phi_{\gamma,\mathsf{e}}(f;r):=\frac{\mathcal{T}_{\gamma,\mathsf{e}}\big(f(B_\mathsf{e}(o,r))\big)}{\mathcal{T}_{\gamma,\mathsf{e}}(B_\mathsf{e}(o,r)}
$$
is strictly increasing in $(0,1)$ unless $f$ is linear.
Consequently,
$$
\lim_{r\downarrow 0} \Phi_{\gamma,\mathsf{e}}(f;r)\le
\Phi_{\gamma,\mathsf{e}}(f;r)\le\lim_{r\uparrow 1}
\Phi_{\gamma,\mathsf{e}}(f;r)\quad\forall\quad r\in (0,1)
$$
holds with equalities when $f$ is linear.
\end{lem}
\begin{proof} The argument for the monotonicity of $\mathcal
Q_{\gamma,\mathsf{e}}(f;r)$ in $(0,1)$ is similar to that proving
\cite[Theorem 1]{CarRa}. The key point is to construct a proper
vector field via the given conformal map $f$. More precisely, if $g$
stands for the inverse map of $f$, then
$$
\xi=\xi(t,w)=f\big((1+{t}{r^{-1}})g(w)\big)\quad\forall\quad w\in
f(B_\mathsf{e}(o,r))
$$
and
$$
\Xi(\xi)=\frac{g(\xi)f'\big(g(\xi)\big)}{r+t}
$$
are selected for (\ref{e3}). Note that the unit outward normal
vector to the boundary $\partial f(B_\mathsf{e}(o,r))$ at $\xi$ is
$$
\nu(\xi)=\Big(\frac{g(\xi)}{r}\Big)\left(\frac{f'\big(g(\xi)\big)}{|f'\big(g(\xi)\big)|_\mathsf{e}}\right)
$$
and so that
$$
\langle\Xi,\nu\rangle_\mathsf{e}=|f'(g(\xi))|_\mathsf{e}\quad\forall\quad
\xi\in \partial f(B_\mathsf{e}(o,r)).
$$

Next, suppose that $u_r$ is the solution of (\ref{eq1}) with
$O=f(B_\mathsf{e}(o,r))$. Then the chain rule yields
$$
|\nabla_\mathsf{e}u_r(\xi)|_\mathsf{e}=\big|\nabla_\mathsf{e}u_r\big(f(z)\big)\big|_\mathsf{e}|f'(z)|_\mathsf{e}\quad\forall\quad
\xi=f(z)\in f(B_\mathsf{e}(o,r)),
$$
whence giving (by Proposition \ref{pr3}):
\begin{equation}\label{eee1}
\frac{d}{dr}\mathcal{T}_{\gamma,\mathsf{e}}\big(f(B_\mathsf{e}(o,r))\big)=\Big(\frac{1+\gamma}{1-\gamma}\Big)\int_{\partial
B_\mathsf{e}(o,r)}|\nabla_\mathsf{e}
u_r|_\mathsf{e}^2\,dL_\mathsf{e}.
\end{equation}
Meanwhile, Proposition \ref{pr2} plus Cauchy-Schwarz's inequality
derives
\begin{equation}\label{eee2}
\mathcal{T}_{\gamma,\mathsf{e}}\big(f(B_\mathsf{e}(o,r))\big)\le\Big(\frac{1+\gamma}{4r^{-1}}\Big)\int_{\partial
B_\mathsf{e}(o,r)}|\nabla_\mathsf{e}
u_r|_\mathsf{e}^2\,dL_\mathsf{e}.
\end{equation}

Finally, putting (\ref{an}), (\ref{eee1}) and (\ref{eee2}) together,
we get that $\frac{d}{dr}\mathcal Q_{\gamma,\mathsf{e}}(f;r)\ge 0$
holds with the strict inequality unless $f$ is linear, whence
reaching the desired result. Since the consequence part is
straightforward, our proof is complete.

\end{proof}

\begin{rem} Lemma \ref{pr4} can be extended to a slightly general form:
For a holomorphic map $f$ from $B_\mathsf{e}(o,1)$ into $\mathbb
R^2$, let ${\bf f}(B_\mathsf{e}(o,r))$ be its Riemann surface with
constant Gauss curvature $-1$. Then
$$
r\mapsto\frac{\mathcal{T}_{\gamma,\mathsf{e}}\big({\bf
f}(B_\mathsf{e}(o,r))\big)}{\mathcal{T}_{\gamma,\mathsf{e}}(B_\mathsf{e}(o,r))}
$$
is strictly increasing in $(0,1)$ unless $f$ is linear. This is in
contrast to \cite[Corollary 2]{CarRa} which reads as:
$$
r\mapsto\frac{\Lambda_{\mathsf{e}}\big({\bf
f}(B_\mathsf{e}(o,r))\big)}{\Lambda_{\mathsf{e}}(B_\mathsf{e}(o,r))}
$$
is strictly decreasing in $(0,1)$ unless $f$ is linear.
\end{rem}

\end{document}